\documentclass{amsart}
\usepackage{amsfonts}
\usepackage{graphicx}

\newtheorem{theorem}{Theorem}[section]
\newtheorem*{theorem A}{Theorem A}
\newtheorem*{theorem B}{N\"olker's Theorem}

\theoremstyle{remark}

\theoremstyle{remark}

\theoremstyle{definition}

\numberwithin{equation}{section}
\def\({\left ( }
\def\){\right )}
\def\<{\left < }
\def\>{\right >}

\input{tcilatex}
\usepackage{rotating}

\begin{document}
\title[FACTORABLE SURFACES IN PSEUDO-GALILEAN
SPACE] {NON-ZERO CONSTANT CURVATURE FACTORABLE SURFACES IN PSEUDO-GALILEAN
SPACE}
\author{Muhittin Evren Aydin$^1$, Mihriban Kulahci $^2$, Alper Osman Ogrenmis$^3$}
\address{$^{1,2,3}$ Department of Mathematics, Faculty of Science, Firat University, Elazig, 23200, Turkey}
\email{meaydin@firat.edu.tr, mihribankulahci@gmail.com, aogrenmis@firat.edu.tr}
\thanks{}

\begin{abstract}
Factorable surfaces, i.e. graphs associated with the product of two functions of one variable, constitute a wide class of surfaces. Such surfaces in the pseudo-Galilean space with zero Gaussian and mean curvature were obtained in \cite {1}. In this study, we provide new classification results relating to the factorable surfaces with non-zero Gaussian and mean curvature. 
\end{abstract}

\maketitle

\section{Introduction}

One of challenging problems in classical differential geometry has
been obtaining surfaces with prescribed Gaussian $(K)$ and mean curvature $(H)$. Let $z=z\left( x,y\right) $ be a
real-valued function of two independent variables. In particular; for the immersed graph of $z$
into a Euclidean space $\mathbb{E}^{3},$ such a problem is reduced to solve
the \textit{Monge-Amp\`{e}re equation} given by (\cite{25,28})
\begin{equation*}
\det \left( \frac{\partial z}{\partial u_{i}\partial u_{j}}\right) =K\left( 1+\left\vert \nabla z\right\vert ^{2} \right )^2, \text{ } u_{1}=x, \text{ }u_{2}=y 
\end{equation*} 
and the \textit{equation of mean curvature type} in
divergence form
\begin{equation*}
 \text{div} \left( \frac{\nabla z}{\sqrt{1+\left\vert
\nabla z\right\vert ^{2}}}\right) =H, 
\end{equation*} 
where $\nabla $ denotes the gradient
of $\mathbb{E}^{2}$ (\cite{18,26,27}). These equations are also related to the branches such as economics, meteorology, oceanography etc. \cite{4}-\cite{8}.

Recall that the graph surfaces are also known as \textit{Monge surfaces} (see \cite{14}, p. 398). In this
study, we deal with a special Monge surface, namely \textit{factorable
surface} that is graph of the function $z\left( x,y\right) =f\left( x\right)
g\left( y\right) $. Such surfaces with $K,H=const.$ in various ambient spaces have been
classified in \cite{3,13,15,17,19,29,32,33}. Our purpose is to analyze the factorable surfaces
in the pseudo-Galilean space $\mathbb{G}_{3}^{1}$ that is one of real
Cayley-Klein spaces (for details, see \cite{12,16,24,30}). As distinct from
the other ambient spaces, there exist two different kinds of factorable surfaces
arising from the absolute figure of $\mathbb{G}_{3}^{1}.$ Explicitly, a
Monge surface in $\mathbb{G}_{3}^{1}$ is said to be \textit{factorable} if
it is given in one of the explicit forms%
\begin{equation*}
\Omega _{1}:z\left( x,y\right) =f\left( x\right) g\left( y\right) \text{ and 
}\Omega _{2}:x\left( y,z\right) =f\left( y\right) g\left( z\right) .
\end{equation*}%
We call $\Omega _{1}$ and $\Omega _{2}$ the \textit{factorable surface of
first }and\textit{\ second kind}, respectively. Note that these surfaces
have different geometric structures in $\mathbb{G}_{3}^{1}$ (such as metric,
curvature etc.). Flat and minimal ($K,H=0$) factorable surfaces in $\mathbb{G}_{3}^{1}$
were presented in \cite{1}. Still, it is an open problem to obtain such surfaces
with $K,H=const. \neq 0$. The present paper is
devoted to solve this problem.

\section{Preliminaries}

In this section, some basics of the Galilean geometry shall be provided from \cite%
{2}, \cite{9}-\cite{11}, \cite{20,21,31}. In particular, the local theory of
immersed surfaces into a pseudo-Galilean space was well-structured in \cite%
{22}.

Let $P_{3}\left( \mathbb{R}\right) $ denote the projective 3-space and $%
\left( u_{0}:u_{1}:u_{2}:u_{3}\right) $ the homogeneous coordinates in $%
P_{3}\left( \mathbb{R}\right) .$ The \textit{pseudo-Galilean 3-space} $\mathbb{G}%
_{3}^{1}$ is a metric space constructed within $P_{3}\left( \mathbb{R}%
\right) $ having the absolute figure $\left\{ \sigma ,l,\epsilon \right\} ,$
where $\sigma $ implies the \textit{absolute plane} of $\mathbb{G}_{3}^{1}$, $l$ \textit{absolute line} in $\sigma $ and $\epsilon $ is the%
\textit{\ hyperbolic involution} of the points of $l$. These arguments are
given by $\sigma :u_{0}=0,$ $l :u_{0}=u_{1}=0$ and 
\begin{equation*}
\epsilon :\left(
u_{0}:u_{1}:u_{2}:u_{3}\right) \longmapsto \left(
u_{0}:u_{1}:u_{3}:u_{2}\right) .
\end{equation*}
The affine model of $\mathbb{G}_{3}^{1}$ can be introduced by changing homogenous coordinates with affine coordinates:
\begin{equation*}
\left(
u_{0}:u_{1}:u_{2}:u_{3}\right) =\left( 1:x:y:z\right) .
\end{equation*} 
In terms of the affine coordinates, the \textit{group of motions} is defined by%
\begin{equation}
\left\{ 
\begin{array}{l}
x^{\prime }=a_{1}+x, \\ 
y^{\prime }=a_{2}+a_{3}x+\left( \cosh \theta \right) y+\left( \sinh \theta
\right) z, \\ 
z^{\prime }=a_{4}+a_{5}x+\left( \sinh \theta \right) y+\left( \cosh \theta
\right) z,%
\end{array}%
\right.   \tag{2.1}
\end{equation}%
where $a_{i},$ $i\in \left\{ 1,...,5\right\} $ and $\theta $ are some
constants. The \textit{pseudo-Galilean distance} is introduced with respect
to the absolute figure, namely 
\begin{equation*}
d\left( x,y\right) =\left\{ 
\begin{array}{ll}
\left\vert x_{2}-x_{1}\right\vert , & \text{if }x_{1} \neq x_{2}, \\ 
\sqrt{\left\vert \left( y_{2}-y_{1}\right) ^{2}-\left( z_{2}-z_{1}\right)
^{2}\right\vert }, & \text{if }x_{1}=x_{2},%
\end{array}%
\right. 
\end{equation*}%
where $x=\left( x_{1},y_{1},z_{1}\right) $ and $y=\left(
x_{2},y_{2},z_{2}\right) .$ Note that this metric (also the absolute figure) is invariant under (2.1).

A plane is said to be \textit{pseudo-Euclidean} if it satisfies the equation 
$x=const.$ Otherwise, it is called \textit{isotropic plane}. A pseudo-Euclidean plane basically has Minkowskian metric while an isotropic plane has Galilean metric, i.e. parabolic measures of distances and angles. Contrary to its denotation, the \textit{isotropic vectors} are contained in the pseudo-Euclidean
plane $x=0$ and, up to the induced Minkowskian metric on this plane, such
vectors are categorized by their causal characters, i.e. \textit{spacelike}, 
\textit{timelike} and \textit{lightlike}. For further details of the Minkowskian geometry, see \cite {23}.

An immersed surface into $\mathbb{G}_{3}^{1}$\ is given by the mapping%
\begin{equation*}
r:D\subseteq \mathbb{R}^{2}\longrightarrow \mathbb{G}_{3}^{1},\text{ }\left(
u_{1},u_{2}\right) \longmapsto \left( x\left( u_{1},u_{2}\right) ,y\left(
u_{1},u_{2}\right) ,z\left( u_{1},u_{2}\right) \right) 
\end{equation*}%
and such a surface is said to be \textit{admissible} (i.e without pseudo-Euclidean tangent
plane) if $x_{,i}=\frac{\partial x}{\partial u_{i}}\neq 0$ for some $i=1,2.$ The first fundamental form is
given by%
\begin{equation*}
ds^{2}=\left( \mathfrak{g}_{1}du_{1}+\mathfrak{g}_{2}du_{2}\right)
^{2}+\omega \left( \mathfrak{h}_{11}du_{1}^{2}+2\mathfrak{h}%
_{12}du_{1}du_{2}+\mathfrak{h}_{22}du_{2}^{2}\right) ,
\end{equation*}%
where $\mathfrak{g}_{i}=x_{,i}$, $\mathfrak{h}%
_{ij}=y_{,i}y_{,j}+z_{,i}z_{,j},$ $i,j=1,2,$ and 
\begin{equation*}
\omega =\left\{ 
\begin{array}{ll}
0, & \text{if }du_{1}:du_{2}\text{ is non-isotropic direction,} \\ 
1, & \text{if }du_{1}:du_{2}\text{ is isotropic direction.}%
\end{array}%
\right. 
\end{equation*}

A side tangent vector field in the tangent plane of the surface $r$ is of the form $x_{,1}r_{,2}-x_{,2}r_{,1}$. Its pseudo-Galilean norm corresponds to
\begin{equation*}
W=\sqrt{\left\vert \left( x_{,1}y_{,2}-x_{,2}y_{,1}\right) ^{2}-\left(
x_{,1}z_{,2}-x_{,2}z_{,1}\right) ^{2}\right\vert }.
\end{equation*}
A surface with $W=0$ is said to be \textit{lightlike}. Throughout the study, all immersed admissible surfaces shall be assumed to be non-lightlike. Then the vector given by%
\begin{equation*}
S=\frac{x_{,1}r_{,2}-x_{,2}r_{,1}}{W}=\frac{1}{W}\left(
0,x_{,1}y_{,2}-x_{,2}y_{,1},x_{,1}z_{,2}-x_{,2}z_{,1}\right) ,
\end{equation*}%
satisfies $S\cdot S=\varepsilon =\left\{ -1,1\right\} ,$ where $"\cdot "$
denotes the Minkowskian scalar product. Hence a surface is said to be 
\textit{spacelike} (\textit{timelike}) if $\varepsilon =1$ ($\varepsilon =-1$%
). The normal vector field is defined as%
\begin{equation*}
N=\frac{1}{W}\left(
0,x_{,1}z_{,2}-x_{,2}z_{,1},x_{,1}y_{,2}-x_{,2}y_{,1}\right) 
\end{equation*}%
such that $N\cdot N=-\varepsilon .$ The second fundamental form is $%
II=\sum_{i,j=1}^{2}L_{ij}du_{i}du_{j},$ where if $\mathfrak{g}_{1}\neq 0$ 
\begin{equation*}
L_{ij}=\frac{\varepsilon }{\mathfrak{g}_{1}}\left( \mathfrak{g}_{1}\left(
0,y_{,ij},z_{,ij}\right) -\mathfrak{g}_{i,j}\left( 0,y_{,1},z_{,1}\right)
\right) \cdot N,\text{ }
\end{equation*}%
otherwise%
\begin{equation*}
L_{ij}=\frac{\varepsilon }{\mathfrak{g}_{2}}\left( \mathfrak{g}_{2}\left(
0,y_{,ij},z_{,ij}\right) -\mathfrak{g}_{i,j}\left( 0,y_{,2},z_{,2}\right)
\right) \cdot N
\end{equation*}%
for $y_{,ij}=\frac{\partial ^{2}y}{\partial u_{i}\partial u_{j}}$, $1\leq i,j\leq 2.$ Consequently, the \textit{Gaussian }and \textit{mean curvature} are defined
as%
\begin{equation*}
K=-\varepsilon \frac{L_{11}L_{22}-L_{12}^{2}}{W^{2}}\text{ and }%
H=-\varepsilon \frac{\mathfrak{g}_{2}^{2}L_{11}-2\mathfrak{g}_{1}\mathfrak{g}_{2}L_{12}+\mathfrak{g}_{1}^{2}L_{22}}{%
2W^{2}}.
\end{equation*}%
A surface is said to have \textit{constant Gaussian} (resp. \textit{mean}) 
\textit{curvature} if $K$ (resp. $H$) is a constant function identically. In particular, it is said to be \textit{%
flat }(resp. \textit{minimal}) if the constant function vanishes.

\section{Factorable Surfaces of First Kind}

Let us consider the factorable surface of first kind in $\mathbb{G}_{3}^{1}$
given in explicit form $\Omega _{1}:z\left( x,y\right) =f\left( x\right)
g\left( y\right) .$ Our purpose is to describe the surface whose $K=const.\neq0$ and $H=const.\neq0$. For
this, firstly we can give the following result:

\begin{theorem} Let a factorable surface of first kind in $\mathbb{G}%
_{3}^{1}$ have non-zero constant Gaussian curvature $K_{0}.$ Then we have:
\begin{equation*}
z\left( x,y\right) =\pm \tanh \left( \sqrt{\left \vert K_{0} \right \vert }x+\lambda _{1}\right)
\left( y+\lambda _{2}\right) ,\text{ }\lambda _{1},\lambda _{2}\in \mathbb{R}.
\end{equation*}
\end{theorem}

\begin{proof} Assume that $\Omega_{1}$ has
non-zero constant Gaussian curvature $K_{0}$. Hence, we get a relation as
follows:%
\begin{equation}
K_{0}= \frac{fgf^{\prime \prime }g^{\prime \prime }-\left(
f^{\prime }g^{\prime }\right) ^{2}}{\left[ 1-\left( fg^{\prime }\right) ^{2}%
\right] ^{2}},  \tag{3.1}
\end{equation}%
where $f^{\prime }=\dfrac{df}{dx},$ $g^{\prime }=\dfrac{dg}{dy},$ etc. $%
K_{0} $ vanishes identically when $f$ or $g$ is a constant function. Then $f$
and $g$ must be non-constant functions. We distinguish two cases for the equation (3.1):

\begin{itemize}
\item[\textbf{Case a.}]
$f^{\prime }=f_{0},$ $f_{0}\in \mathbb{R}-\left\{ 0\right\}
.$ Thereby (3.1) turns into the following\ polynomial equation on $\left(
g^{\prime }\right) $:%
\begin{equation*}
K_{0}+\left( f_{0}^{2}-2K_{0}f^{2}\right) \left( g^{\prime }\right)
^{2}+K_{0}f^{4}\left( g^{\prime }\right) ^{4}=0,
\end{equation*}%
which yields a contradiction.

\item[\textbf{Case b.}] $f^{\prime \prime }\neq 0.$ We have again two cases:

\begin{itemize}
\item[\textbf{Case b.1. }] $g^{\prime }=g_{0},$ $g_{0}\in \mathbb{R}-\left\{
0\right\} .$ Then (3.1) leads to%
\begin{equation}
\pm \sqrt{\left \vert K_{0} \right \vert}=\frac{g_{0}f^{\prime }}{1-\left( g_{0}f\right)^{2}}.  \tag{3.2}
\end{equation}%
After solving (3.2), we obtain%
\begin{equation*}
f\left( x\right) =\pm \frac{1}{g_{0}}\tanh \left( \sqrt{\left \vert K_{0} \right \vert}x+\lambda
_{1}\right) , \text{ }\lambda _{1}\in \mathbb{R}.
\end{equation*}

\item[\textbf{Case b.2. }]$g^{\prime \prime }\neq 0.$ Then (3.1) can be arranged
as the following:%
\begin{equation}
\frac{K_{0}\left[ 1-\left( fg^{\prime }\right) ^{2}\right] ^{2}}{ff^{\prime
\prime }\left( g^{\prime }\right) ^{2}}=\frac{gg^{\prime \prime }}{\left(
g^{\prime }\right) ^{2}}-\frac{\left( f^{\prime }\right) ^{2}}{ff^{\prime
\prime }}.  \tag{3.3}
\end{equation}%
The partial derivative of (3.2) with respect to $x$ and $y$ leads to a
polynomial equation on $\left( g^{\prime }\right) $:%
\begin{equation}
-\left( \dfrac{1}{ff^{\prime \prime }}\right) ^{\prime }+\left( \frac{f^{3}}{%
f^{\prime \prime }}\right) ^{\prime }\left( g^{\prime }\right) ^{4}=0. 
\tag{3.4}
\end{equation}%
Since all coefficients must vanish in (3.4), the contradiction $f^{\prime }=0
$ is obtained. Therefore the proof is completed.
\end{itemize}
\end{itemize}
\end{proof}

\begin{theorem} Let a factorable surface of first kind in $\mathbb{G}%
_{3}^{1}$ have non-zero constant mean curvature $H_{0}.$ Then the following occurs:%
\begin{equation*}
z\left( x,y\right) =f_{0}g\left( y\right) =\frac{1}{2H_{0}}\sqrt{\left(
2H_{0}y+\lambda _{1}\right) ^{2}\pm 1}+\lambda _{2},
\end{equation*}%
where $"\pm "$ happens plus (resp. minus) when the surface is timelike
(resp. spacelike). Further, $f_{0}$ is non-zero constant and $\lambda
_{1},\lambda _{2}$ some constants.
\end{theorem}

\begin{proof} Relating to the mean curvature, we get
\begin{equation}
H_{0}=\frac{fg^{\prime \prime }}{2\left\vert 1-\left( fg^{\prime
}\right) ^{2}\right\vert ^{\frac{3}{2}}}.  \tag{3.5}
\end{equation}%
It is clear from (3.5) that $g$ is
a non-linear function. By taking parital derivative of (3.5) with respect to $%
x,$ we deduce%
\begin{equation}
f^{\prime }\left\vert 1-\left( fg^{\prime }\right) ^{2}\right\vert -\frac{3f}{2}\left\vert -2ff^{\prime }\left( g^{\prime }\right) ^{2}\right\vert
=0, 
\tag{3.6}
\end{equation}%
which yields two cases:
\begin{itemize}
\item[\textbf{Case a.}]
$f=f_{0}\neq 0$, $f_{0}\in \mathbb{R}$, is a solution for (3.6). If the surface is
spacelike, then (3.5) turns to 
\begin{equation}
2H_{0}=\frac{f_{0}g^{\prime \prime }}{\left[ 1-\left( f_{0}g^{\prime }\right)
^{2}\right] ^{\frac{3}{2}}}.  \tag{3.7}
\end{equation}%
By solving (3.7), we find 
\begin{equation*}
g\left( y\right) =\frac{1}{f_{0}H_{0}}\sqrt{\left( 2H_{0}y+\lambda
_{1}\right) ^{2}-1}+\lambda _{2},
\end{equation*}%
where $\lambda _{1}$ and $\lambda _{2}$ are some constants. Otherwise, i.e.
timelike situation yields%
\begin{equation}
2H_{0}=\frac{f_{0}g^{\prime \prime }}{\left[ \left( f_{0}g^{\prime }\right)
^{2}-1\right] ^{\frac{3}{2}}}.  \tag{3.8}
\end{equation}%
After solving (3.8), we obtain%
\begin{equation*}
g\left( y\right) =\frac{1}{f_{0}H_{0}}\sqrt{\left( 2H_{0}y+\lambda
_{3}\right) ^{2}+1}+\lambda _{4}
\end{equation*}%
for some constants $\lambda _{3},\lambda _{4}.$

\item[\textbf{Case b.}] $f^{\prime }\neq 0.$ If the surface is spacelike or
timelike, then (3.6) implies%
\begin{equation*}
1+2\left( fg^{\prime }\right) ^{2}=0,
\end{equation*}%
which is not possible.
\end{itemize}
\end {proof}

\section{Factorable Surfaces of Second Kind}

As in previous section, by assuming $K=const.\neq 0$ and $H=const.\neq 0,$
we try to describe the factorable graph surfaces of second kind in $\mathbb{G%
}_{3}^{1}$ given in explicit form $\Omega _{2}:x\left( y,z\right) =f\left(
y\right) g\left( z\right) .$ Therefore the following non-existence result can be stated:

\begin{theorem} There does not exist a factorable surface of second kind in $%
\mathbb{G}_{3}^{1}$ having non-zero constant Gaussian curvature.
\end{theorem}

\begin{proof} It is proved by contradiction. Then we suppose that $\Omega_{2}$ has the Gaussian curvature $K_{0}\neq 0$ in $\mathbb{G}_{3}^{1}.$ By a calculation, relating to the Gaussian curvature, we get%
\begin{equation}
K_{0}=\frac{fgf^{\prime \prime }g^{\prime \prime }-\left( f^{\prime
}g^{\prime }\right) ^{2}}{\left[ \left( fg^{\prime }\right) ^{2}-\left(
f^{\prime }g\right) ^{2}\right] ^{2}},  \tag{4.1}
\end{equation}%
where $f^{\prime }=\dfrac{df}{dy},$ $g^{\prime }=\dfrac{dg}{dz}$ and so on.
Hereinafter $f$ and $g$ must be non-constant functions so that $K_{0}$ does not
vanish. Point that the roles of $f$ and $g$ are symmetric and it is sufficient to
discuss the cases depending on $f.$ Thus, if $f^{\prime \prime }=0$ i.e. \linebreak $f^{\prime }=f_{0}\neq 0,$ then (4.1) turns a polynomial equation on $(f)$:%
\begin{equation}
\left[ K_{0}\left( g^{\prime }\right) ^{4}\right] f^{4}-\left[ 2K_{0}\left(
f_{0}gg^{\prime }\right) ^{2}\right] f^{2}+\left( f_{0}g\right) ^{4}+\left(
f_{0}g^{\prime }\right) ^{2}=0.  \tag{4.2}
\end{equation}%
The fact that the coefficients must be zero yields the contradiction $g^{\prime }= 0$. Hence $f$ is a non-linear function and, by symmetry, so is $g$. By dividing (4.1) with $ff^{\prime \prime }\left( g^{\prime
}\right) ^{2},$ we can write%
\begin{equation}
K_{0}\left[ \frac{f^{3}}{f^{\prime \prime }}\left( g^{\prime }\right) ^{2}-2%
\frac{f\left( f^{\prime }\right) ^{2}}{f^{\prime \prime }}g^{2}+\frac{\left(
f^{\prime }\right) ^{4}}{ff^{\prime \prime }}\left( \frac{g^{2}}{g^{\prime }}%
\right) ^{2}\right] =\frac{gg^{\prime \prime }}{\left( g^{\prime }\right)
^{2}}-\frac{\left( f^{\prime }\right) ^{2}}{ff^{\prime \prime }}.  \tag{4.3}
\end{equation}%
Put $f^{\prime }=p,$ $\dot{p}=\dfrac{dp}{df}=\dfrac{f^{\prime \prime }}{%
f^{\prime }}$ and $g^{\prime }=r,$ $\dot{r}=\dfrac{dr}{dg}=\dfrac{g^{\prime
\prime }}{g^{\prime }}$ in (4.3). Then the partial derivative of (4.3) with
respect to $g$ gives%
\begin{equation}
K_{0}\left[ 2\frac{f^{3}}{p\dot{p}}r\dot{r}-4\frac{fp}{\dot{p}}g+2\frac{p^{3}%
}{f\dot{p}}\left( \frac{g^{2}}{r}\right) \left\{ \frac{d}{dg}\left( \frac{%
g^{2}}{r}\right) \right\} \right] =\frac{d}{dg}\left( \frac{g\dot{r}}{r}%
\right) .  \tag{4.4}
\end{equation}%
The partial derivative of (4.4) with respect to $f$ yields%
\begin{equation}
r\dot{r}\left[ \frac{d}{df}\left( \frac{f^{3}}{p\dot{p}}\right) \right] -2g%
\left[ \frac{d}{df}\left( \frac{fp}{\dot{p}}\right) \right] +\left[ \frac{d}{%
df}\left( \frac{p^{3}}{f\dot{p}}\right) \right] \left[ \left( \frac{g^{2}}{r}%
\right) \frac{d}{dg}\left( \frac{g^{2}}{r}\right) \right] =0.  \tag{4.5}
\end{equation}%
By dividing $\left( 4.5\right) $ with $g$ and taking partial derivative with
respect to $g,$ we derive%
\begin{equation}
\underset{F_{1}\left( f\right) }{\underbrace{\left[ \frac{d}{df}\left( \frac{%
f^{3}}{p\dot{p}}\right) \right] }}\overset{G_{1}\left( g\right) }{\overbrace{%
\left[ \frac{d}{dg}\left( \frac{r\dot{r}}{g}\right) \right] }}+\underset{%
F_{2}\left( f\right) }{\underbrace{\left[ \frac{d}{df}\left( \frac{p^{3}}{f%
\dot{p}}\right) \right] }}\overset{G_{2}\left( g\right) }{\overbrace{\left[ 
\frac{d}{dg}\left\{ \left( \frac{g}{r}\right) \frac{d}{dg}\left( \frac{g^{2}%
}{r}\right) \right\} \right] }}=0.  \tag{4.6}
\end{equation}%
We have to distinguish several cases:

\begin{itemize}
\item[\textbf{Case a.}]
$F_{1}=0.$ Then $f^{3}=\lambda _{1}p\dot{p},$ $\lambda _{1} \in \mathbb{R},$ $\lambda
_{1}\neq 0.$ We have again two cases:

\begin{itemize}
\item[\textbf{Case a.1.}]
$F_{2}=0,$ namely $p^{3}=\lambda _{2}f\dot{p},$ $\lambda _{2} \in \mathbb{R},$ $\lambda
_{2}\neq 0.$ Considering these in (4.5) implies $fp=\lambda _{3}\dot{p},$ $\lambda _{3} \in \mathbb{R},$ $%
\lambda _{3}\neq 0.$ Substituting these into $\left( 4.3\right) $ yields%
\begin{equation}
K_{0}\left[ \lambda _{1}r^{2}-2\lambda _{3}g^{2}+\lambda _{2}\left( \frac{%
g^{2}}{r}\right) ^{2}\right] -\frac{g\dot{r}}{r}=\frac{-\lambda _{2}}{p^{2}}.
\tag{4.7}
\end{equation}%
The left side of (4.7) is either a function of $g$ or a constant, however
other side is a non-constant function of $f.$ This is not possible.

\item[\textbf{Case a.2.}] $G_{2}=0.$ It implies $\dfrac{d}{dg}\left( \dfrac{g^{2}}{r}%
\right) =\dfrac{\lambda _{4}r}{g},$ $\lambda _{4} \in \mathbb{R}.$ By considering this one into (4.5)
together with the assumption of Case a, we conclude%
\begin{equation}
\left[ \frac{d}{df}\left( \frac{p}{f}\right) \right] \left[ 1-\lambda
_{4}\left( \frac{p}{f}\right) ^{2}\right] =0.  \tag{4.8}
\end{equation}%
If $p=\lambda _{5}f,$ $\lambda _{5} \in \mathbb{R},$ $\lambda _{5}\neq 0,$ in (4.8) then we have $\dot{p}%
=\lambda _{5}$. Combining it with the assumption of Case a gives $%
f^{2}=\lambda _{1}\lambda _{5}^{2}$ that contradicts with $K_{0}\neq 0.$
\end{itemize}

\item[\textbf{Case b.}] $G_{1}=0.$ Hence $r\dot{r}=\lambda _{1}g,$ $\lambda _{1} \in \mathbb{R},$ $\lambda
_{1}\neq 0.$ We have two cases:

\begin{itemize}
\item[\textbf{Case b.1.}] 
$F_{2}=0,$ i.e. $p^{3}=\lambda _{2}f\dot{p},$ $\lambda _{2} \in \mathbb{R},$ $\lambda
_{2}\neq 0.$ Then (4.5) follows%
\begin{equation}
\left[ \frac{d}{df}\left( \frac{f}{p}\right) \right] \left[ \lambda
_{1}\left( \frac{f}{p}\right) ^{2}-1\right] =0.  \tag{4.9}
\end{equation}%
If $p=\lambda _{3}f,$ $\lambda _{3} \in \mathbb{R},$ $\lambda _{3}\neq 0,$ in (4.8) then we get $\dot{p}%
=\lambda _{3}.$ Comparing this one with the assumption of Case b.1 gives $%
f^{2}=\lambda _{2},$ which is no possible since $K_{0}\neq 0.$

\item[\textbf{Case b.2.}]  $G_{2}=0.$ It follows 
\begin{equation}
\left( \frac{g}{r}\right) \frac{d}{dg}\left( \frac{g^{2}}{r}\right) =\lambda
_{4}, \text{ } \lambda _{4} \in \mathbb{R}.  \tag{4.10}
\end{equation}%
An integration of (4.10) with respect to $g$ gives%
\begin{equation}
r=\frac{g^{2}}{\sqrt{\lambda _{4}g^{2}+\lambda _{5}}},  \text{ } \lambda _{5} \in \mathbb{R}, \tag{4.11}
\end{equation}%
where $\lambda _{4}$ and $\lambda _{5}$ are not equal to zero together.
After taking derivative of (4.11) with respect to $g$ and producting with $r,$ we
conclude%
\begin{equation}
r\dot{r}=\frac{\lambda _{4}g^{5}+2\lambda _{5}g^{3}}{\left( \lambda
_{4}g^{2}+\lambda _{5}\right) ^{2}}.  \tag{4.12}
\end{equation}%
Due to the assumption of the Case b, (4.12) turns to the following
polynomial equation on $g:$%
\begin{equation}
\left( \lambda _{4}-\lambda _{1}\lambda _{4}^{2}\right) g^{5}+2\left( \lambda
_{5}-\lambda _{1}\lambda _{4}\lambda _{5}\right) g^{3}-\left(\lambda _{1}\lambda
_{5}^{2}\right)g=0.  \tag{4.13}
\end{equation}%
Since $\lambda _{1}\neq 0,$ we get $1=\lambda _{1}\lambda _{4}$ and $\lambda _{5}=0.$ It follows from $\left( 4.11\right) $ that $r=\left(\lambda _{4} \right )^{\frac{-1}{2}}g$. Then by substituting it into (4.3), we obtain%
\begin{equation}
K_{0}\left[ \frac{f^{3}}{\lambda _{4}p\dot{p}}-2\frac{fp}{\dot{p}}+\lambda _{4}\frac{p^{3}}{f\dot{p}}%
\right] g^{2}+\frac{p}{f\dot{p}} -1=0.  \tag{4.14}
\end{equation}%
This polynomial equation leads to 
\begin{equation}
f\dot{p}=p  \tag{4.15}
\end{equation}
and 
\begin{equation}
\frac{f^{3}}{\lambda _{4}p\dot{p}}-2\frac{fp}{\dot{p}}+\lambda _{4}\frac{p^{3}}{f\dot{p}}=0 .
\tag{4.16}
\end{equation}%
Substituting (4.15) into (4.16) gives $p=\pm \left(\lambda _{4} \right )^{\frac{-1}{2}} f$ or $f\left( y\right)
=\lambda _{6}\exp \left( \pm y\right) ,$ $\lambda _{6}\in \mathbb{R}$, $\lambda _{6}\neq 0.$ Further, since $r=\left(\lambda _{4} \right )^{\frac{-1}{2}}g$, we
have $g\left( z\right) =\lambda _{7}\exp \left( z\right) ,$ $\lambda _{7}\in \mathbb{R}$, $\lambda
_{7}\neq 0$. However these lead the surface to be flat, i.e. $K_{0}=0$, which
is not our case.
\end{itemize}

\item[\textbf{Case c.}] $F_{1}G_{1}\neq 0.$ Then (4.6) can be rewritten as%
\begin{equation}
\frac{F_{1}\left( f\right) }{F_{2}\left( f\right) }+\frac{G_{2}\left(
g\right) }{G_{1}\left( g\right) }=0,  \tag{4.17}
\end{equation}%
which implies 
\begin{equation}
\frac{f^{3}}{p\dot{p}}=\lambda _{1}\frac{p^{3}}{f\dot{p}}+\lambda _{2},\text{
and\ }\left( \frac{g}{r}\right) \frac{d}{dg}\left( \frac{g^{2}}{r}\right)
=-\lambda _{1}\frac{r\dot{r}}{g}+\lambda _{3},  \tag{4.18}
\end{equation}%
where $\lambda _{1},\lambda _{2},\lambda _{3}\in \mathbb{R}$, $\lambda _{1}\neq 0$. Substituting
(4.18) into (4.5) gives%
\begin{equation}
2\frac{fp}{\dot{p}}=\lambda _{3}\frac{p^{3}}{f\dot{p}}+\lambda _{4}, \text{ } \lambda_{4} \in \mathbb{R}. 
\tag{4.19}
\end{equation}%
Comparing (4.19) with the first equality in (4.18) leads to%
\begin{equation}
\left\{ 
\begin{array}{l}
f^{4}-\lambda _{1}p^{4}=\lambda _{2}fp\dot{p} \\ 
2f^{2}p^{2}-\lambda _{3}p^{4}=\lambda _{4}fp\dot{p}.%
\end{array}%
\right.  \tag{4.20}
\end{equation}%
It is clear that $\lambda _{3}$ and $\lambda _{4}$ cannot be zero together.
Without lose of generality, we may assume $\lambda _{4}=1.$ Then (4.20)
turns to%
\begin{equation}
f^{4}-2\lambda _{2}f^{2}p^{2}+\left( \lambda _{2}\lambda _{3}-\lambda
_{1}\right) p^{4}=0  \tag{4.21}
\end{equation}%
or%
\begin{equation}
\left( \frac{f}{p}\right) ^{2}+\left( \lambda _{2}\lambda _{3}-\lambda
_{1}\right) \left( \frac{f}{p}\right) ^{-2}=2\lambda _{2}.  \tag{4.22}
\end{equation}%
Taking derivative of (4.22) with respect to $f$ leads to%
\begin{equation}
\frac{d}{df}\left( \frac{f}{p}\right) \left[ 1-\left( \lambda _{2}\lambda
_{3}-\lambda _{1}\right) \left( \frac{f}{p}\right) ^{-4}\right] =0. 
\tag{4.23}
\end{equation}%
If the ratio $f/p$ is constant, i.e. $p=\lambda _{5}f,$ $\lambda _{5} \in \mathbb{R},$ $\lambda _{5}\neq 0,$
then the second equality in (4.20) implies%
\begin{equation*}
\left( 2\lambda _{5}^{2}-\lambda _{3}\lambda _{5}^{4}\right) f^{4}=\lambda _{5}^{2}f^{2},
\end{equation*}%
which is not possible. This completes the proof.
\end{itemize}
\end {proof}

\begin{theorem} Let a factorable surface of second kind in $\mathbb{G}%
_{3}^{1}$ have non-zero constant mean curvature $H_{0}.$ Then we have:%
\begin{equation*}
x\left( y,z\right) =\lambda _{1}\exp \left(\lambda _{2}y+ \frac{\lambda _{2}}{2H_{0}}\sqrt{%
\left( 2H_{0}z+\lambda _{3}\right) ^{2}\pm 1}\right) ,
\end{equation*}%
where $"\pm "$ happens plus (resp. minus) when the surface is timelike
(resp. spacelike). Further, $\lambda _{1},\lambda _{2}$ are
non-zero constans and $\lambda _{3}$ some constant.
\end{theorem} 

\begin{proof} The proof is only done for spacelike situation since the calculations are almost same for other situation. Then we have
\begin{equation*}
\left( fg^{\prime }\right) ^{2}-\left( f^{\prime }g\right) ^{2}>0
\end{equation*}%
for all pairs $\left( y,z\right) .$ Since the mean curvature is constant $%
H_{0}\neq 0,$ by a calculation, we deduce
\begin{equation}
2H_{0}\left[ \left( fg^{\prime }\right) ^{2}-\left( f^{\prime
}g\right) ^{2} \right] ^{\frac{3}{2}}=\left( fg^{\prime }\right)
^{2}f^{\prime \prime }g-2fg\left( f^{\prime }g^{\prime }\right) ^{2}+\left(
f^{\prime }g\right) ^{2}fg^{\prime \prime }.  \tag{4.24}
\end{equation}%
Note that $f$ is not a constant function since $H_{0}\neq 0$ and, by
symmetry, neither is $g.$ Then dividing (4.24) with $fg\left( f^{\prime
}g^{\prime }\right) ^{2}$ yields%
\begin{equation}
2H_{0}\left[ \left( \frac{f}{f^{\prime }}\right) ^{\frac{4}{3}}\left( \frac{%
g^{\prime }}{g}\right) ^{\frac{2}{3}}-\left( \frac{f^{\prime }}{f}\right)
^{\frac{2}{3}}\left( \frac{g}{g^{\prime }}\right) ^{\frac{4}{3}}\right] ^{\frac{3}{2}}=%
\frac{ff^{\prime \prime }}{\left( f^{\prime }\right) ^{2}}+\frac{gg^{\prime
\prime }}{\left( g^{\prime }\right) ^{2}}-2.  \tag{4.25}
\end{equation}%
Let us put $f^{\prime }=p,$ $\dot{p}=\dfrac{dp}{df}=\dfrac{f^{\prime \prime }%
}{f^{\prime }}$ and $g^{\prime }=r,$ $\dot{r}=\dfrac{dr}{dg}=\dfrac{%
g^{\prime \prime }}{g^{\prime }}$ in (4.25). Thus (4.25) can be rewritten as%
\begin{equation}
2H_{0}\left[ \left( \frac{f}{p}\right) ^{\frac{4}{3}}\left( \frac{r}{g}%
\right) ^{\frac{2}{3}}-\left( \frac{p}{f}\right) ^{\frac{2}{3}}\left( \frac{g%
}{r}\right) ^{\frac{4}{3}}\right] ^{\frac{3}{2}}=\frac{f\dot{p}}{p}+\frac{g%
\dot{r}}{r}-2.  \tag{4.26}
\end{equation}%
The partial derivative of (4.26) with respect to $f$ gives%
\begin{equation}
\left. 
\begin{array}{l}
2H_{0}\left[ \left( \frac{f}{p}\right) ^{\frac{4}{3}}\left( \frac{r}{g}%
\right) ^{\frac{2}{3}}-\left( \frac{p}{f}\right) ^{\frac{2}{3}}\left( \frac{g%
}{r}\right) ^{\frac{4}{3}}\right] ^{\frac{1}{2}}\left[ 2\left( \frac{r}{g}%
\right) ^{\frac{2}{3}}+\left( \frac{p}{f}\right) ^{2}\left( \frac{g}{r}%
\right) ^{\frac{4}{3}}\right] \frac{d}{df}\left( \frac{f}{p}\right) = \\ 
=\frac{d}{df}\left( \frac{f\dot{p}}{p}\right) \left( \frac{p}{f}\right) ^{%
\frac{1}{3}}.%
\end{array}%
\right.   \tag{4.27}
\end{equation}%
If $\dot{p}=0$, then (4.27) reduces to%
\begin{equation*}
2+\left( \frac{p}{f}\right) ^{2}\left( \frac{g}{r}\right) ^{2}=0,
\end{equation*}%
which is not possible. Thus $p$ is not constant function and, by
symmetry, so is $r$. In addition, we have to consider two cases in order to solve (4.27):

\begin{itemize}
\item[\textbf{Case a.}] 
$p=\lambda _{1}f,$ $\lambda _{1}\neq 0,$ is a solution for
(4.27). Substituting this one into (4.26) gives%
\begin{equation*}
2H_{0}\left[ \lambda _{1}^{-\frac{4}{3}}\left( \frac{r}{g}\right) ^{\frac{2}{%
3}}-\lambda _{1}^{\frac{2}{3}}\left( \frac{g}{r}\right) ^{\frac{4}{3}}\right]
^{\frac{3}{2}}=\frac{g\dot{r}}{r}-1
\end{equation*}%
or%
\begin{equation*}
2H_{0}\left[ \left( \frac{r}{g}\right) ^{2}-\lambda _{1}^{2}\right] ^{\frac{3%
}{2}}=\lambda _{1}^{2}\left[ \frac{r\dot{r}}{g}-\left( \frac{r}{g}\right)
^{2}\right] .
\end{equation*}%
The last equality can be rearranged as%
\begin{equation}
2H_{0}=\frac{\lambda _{1}^{2}\left( \dfrac{g^{\prime }}{g}\right) ^{\prime }%
}{\left[ \left( \dfrac{g^{\prime }}{g}\right) ^{2}-\lambda _{1}^{2}\right] ^{%
\frac{3}{2}}}.  \tag{4.28}
\end{equation}%
An integration of (4.28) with respect to $z$ yields%
\begin{equation*}
2H_{0}z+\lambda _{2}=\frac{-\dfrac{g^{\prime }}{g}}{\sqrt{\left( \dfrac{%
g^{\prime }}{g}\right) ^{2}-\lambda _{1}^{2}}}
\end{equation*}%
or%
\begin{equation}
\dfrac{g^{\prime }}{g}=\frac{\lambda _{1}\left( 2H_{0}z+\lambda _{2}\right) 
}{\sqrt{\left( 2H_{0}z+\lambda _{2}\right) ^{2}-1}}.  \tag{4.29}
\end{equation}%
An again integration of (4.29) with respect to $z$ leads to%
\begin{equation*}
g\left( z\right) =\lambda _{3}\exp \left( \frac{\lambda _{1}}{2H_{0}}\sqrt{%
\left( 2H_{0}z+\lambda _{2}\right) ^{2}-1}\right) ,\text{ }\lambda _{3}\neq
0.
\end{equation*}%
Due to the assumption of Case a, we conclude $f\left( y\right) =\lambda
_{4}\exp \left( \lambda _{1}y\right) ,$ $\lambda _{4}\neq 0$, which gives
the assertion of the theorem.

\item[\textbf{Case b.}] $\dfrac{d}{df}\left( \dfrac{f}{p}\right) \neq 0.$ By
symmetry, we deduce $\dfrac{d}{dg}\left( \dfrac{g}{r}\right) \neq 0.$ Then
(4.27) can be rewritten as%
\begin{equation}
\left[ \left( \frac{r}{g}\right) ^{2}-\left( \frac{p}{f}\right) ^{2}\right]
^{\frac{1}{2}}\left[ 2+\left( \frac{p}{f}\right) ^{2}\left( \frac{r}{g}%
\right) ^{-2}\right] =\frac{\frac{d}{df}\left( \dfrac{f\dot{p}}{p}\right)
\left( \dfrac{p}{f}\right) }{2H_{0}\frac{d}{df}\left( \dfrac{f}{p}\right) }.
\tag{4.30}
\end{equation}%
The partial derivative of (4.30) with respect to $g$ leads to%
\begin{equation}
2\left( \frac{r}{g}\right) ^{4}-\left( \frac{r}{g}\right) ^{2}\left( \frac{p%
}{f}\right) ^{2}+2\left( \frac{p}{f}\right) ^{4}=0.  \tag{4.31}
\end{equation}%
Since the ratios $\dfrac{p}{f}$ and $\dfrac{r}{g}$ are not constant, (4.31)
presents a polynomial equation which yields a contradiction. Therefore the proof is completed.
\end{itemize}
\end{proof}

\end{document}